\documentclass{amsart}

\usepackage{amsmath,amssymb,url}
\usepackage{enumitem}
\usepackage{hyperref}

\theoremstyle{plain}
\newtheorem{Thm}{Theorem}[section] 
 
\newtheorem{Lem}[Thm]{Lemma} 
\newtheorem{Cor}[Thm]{Corollary} 

\theoremstyle{definition}
\newtheorem{Def}[Thm]{Definition}
\newtheorem{Rem}[Thm]{Remark}
\numberwithin{equation}{section} 

\newcommand{\Dep}{\mathrm{Dep}}
\newcommand{\KK}{{\mathbb K}}
\newcommand{\Cid}{\mathrm{Clg}}
\newcommand{\ZZ}{{\mathbb{Z}}}
\newcommand{\NN}{{\mathbb{N}}}
\newcommand{\FF}{{\mathbb{F}}}
\DeclareMathAlphabet\mathbfsl {T1}{cmr}{bx}{it}

\begin{document}
	
	\title[Closed sets of functions]{Closed sets of finitary functions between products of finite fields of coprime order}
	\author{Stefano Fioravanti}
	
	
	\address{
		Institut f\"ur Algebra,
		Johannes Kepler Universit\"at Linz,
		4040 Linz,
		Austria}
	\email{\tt stefano.fioravanti66@gmail.com}

	\thanks{The research was supported by the Austrian Science Fund (FWF):P29931.}
	

	\begin{abstract}
		
		We investigate the finitary functions from a finite product of finite fields $\prod_{j =1}^m\mathbb{F}_{q_j} = \KK$ to a finite product of finite fields $\prod_{i =1}^n\mathbb{F}_{p_i} = \FF$, where $|\KK|$ and $|\FF|$ are coprime. An $(\FF,\KK)$-linearly closed clonoid is a subset of these functions which is closed under composition from the right and from the left with linear mappings. 
		
		We give a characterization of these subsets of functions through the $\FF_p[\KK^{\times}]$-submodules of $\mathbb{F}_p^{\KK}$, where $\KK^{\times}$ is the multiplicative monoid of $\KK = \prod_{i=1}^m \FF_{q_i}$. Furthermore we prove that each of these subsets of functions is generated by a set of unary functions and we provide an upper bound for the number of distinct $(\FF,\KK)$-linearly closed clonoids.
		
	\end{abstract}
	
	\maketitle
	
	\section{Introduction}
	
	Since P. Hall's abstract definition of a clone the problem to describe sets of finitary functions from a set $A$ to a set $B$  which satisfy some closure properties has been a fruitful branch of research. E. Post's characterization of all clones on a two-element set \cite{Pos.TTVI} can be considered as a foundational result in this field, which was developed further, e. g., in \cite{Ros.MCOA,PK.FUR,Sze.CIUA,Leh.CCOF}. Starting from \cite{BJK.TCOC}, clones are used to study the complexity of certain constrain satisfaction problems (CSPs).
	
	The aim of this paper is to describe sets of functions from a finite product of finite fields $\prod_{j =1}^m\mathbb{F}_{q_j} = \KK$ to a finite product of finite fields $\prod_{i =1}^n\mathbb{F}_{p_i} = \FF$, where $|\KK|$ and $|\FF|$ are coprime. The sets of functions we are interested in are closed under composition from the left and from the right with linear mappings. Thus we consider sets of functions with different domains and codomains; such sets are called clonoids and are investigated, e. g., in \cite{AM.FGEC}. Let $\mathbf{B}$ be an algebra, and let $A$ be a non-empty set. For a subset $C$ of $\bigcup_{n \in \NN} B^{A^n}$ and $k\in \NN$, we let $C^{[k]} :=C \cap B^{A^k}$. According to Definition $4.1$ of \cite{AM.FGEC} we call $C$ a \emph{clonoid} with source set $A$ and target algebra $\mathbf{B}$ if
	\begin{center}
		\begin{enumerate}
			\item [(1)] for all $k \in \NN$: $C^{[k]}$ is a subuniverse of $\mathbf{B}^{A^k}$, and
			\item [(2)] for all $k,n \in \NN$, for all $(i_1,\dots,i_k) \in \{1,\dots,n\}^k$, and for all $c \in C^{[k]}$, the function $c' \colon A^n \to B$ with $c'(a_1,\dots,a_n) := c(a_{i_1},\dots,a_{i_k})$ lies in $C^{[n]}$.
		\end{enumerate}
	\end{center}
	
	By $(1)$ every clonoid is closed under composition with operations of $\mathbf{B}$ on the left. In particular we are dealing with those clonoids whose target algebra is the ring $\prod_{i=1}^m\mathbb{F}_{p_i}$ that are closed under composition with linear mappings from the right side.
	
	\begin{Def}
		\theoremstyle{definition}
		\label{DefClo-2}
		Let $m,s \in \NN$  and let $\KK= \prod_{j=1}^m\mathbb{K}_{j}$, $\FF= \prod_{i=1}^s\mathbb{F}_{i}$ be products of fields. An \emph{$(\FF,\KK)$-linearly closed clonoid} is a non-empty subset $C$ of $\bigcup_{k \in \mathbb{N}} \prod_{i=1}^s\mathbb{F}_{i}^{{\prod_{j=1}^m\mathbb{K}_{j}^k}}$ with the following properties:
		
		\begin{enumerate}
			\item[(1)] for all $n \in \NN$, $\mathbfsl{a}, \mathbfsl{b} \in \prod_{i=1}^s\mathbb{F}_i$, and $f,g \in C^{[n]}$:
			
			\begin{equation*}
				\mathbfsl{a}f + \mathbfsl{b}g \in C^{[n]};
			\end{equation*}
			
			\item[(2)] for all $l,n \in \NN$, $f \in C^{[n]}$, $(\mathbfsl{x}_1,\dots,\mathbfsl{x}_m) \in \prod_{j=1}^m\mathbb{K}_{j}^l$, and $A_j\in \mathbb{K}^{n \times l}_{j}$:
			
			\begin{equation*}
				g\colon (\mathbfsl{x}_1,\dots,\mathbfsl{x}_m) \mapsto f(A_1\cdot \mathbfsl{x}_1^t,\cdots,A_m\cdot \mathbfsl{x}_m^t) \text{ is in } C^{[l]},
			\end{equation*}
			
		\end{enumerate}
		where the juxtaposition $\mathbfsl{a}f$ denotes the Hadamard product of the two vectors (i.e. the component-wise product $(a_1,\dots,a_n)\cdot (b_1,\dots,b_n) = (a_1b_1,\dots,$ $a_nb_n)$).
	\end{Def}
	
	Clonoids naturally appear in the study of promise constraint satisfaction problems (PCSPs). These problems are investigated,  e. g., in \cite{BG.PCSS}, and in \cite{BKO.AATP} clonoid theory has been used to provide an algebraic approach to PCSPs. In \cite{Spa.OTNO} A. Sparks investigate the number of clonoids for a finite set $A$ and finite algebra $\mathbf{B}$ closed under the operations of $\mathbf{B}$. In \cite{Kre.CFSO} S. Kreinecker characterized linearly closed clonoids on $\mathbb{Z}_p$, where $p$ is a prime. Furthermore, a description of the set of all $(\FF,\KK)$-linearly closed clonoids is a useful tool to investigate (polynomial) clones on $\ZZ_n$, where $n$ is a product of distinct primes or to represent polynomial functions of semidirect products of groups. 
	
	In \cite{Fio.CSOF} there is a complete description of the structure of all $(\FF,\KK)$-linearly closed clonoids in case $\FF$ and $\KK$ are fields and the results we will present are a generalization of this description.
	
	The main result of this paper (Theorem \ref{Thm14-2}) states that every $(\FF,\KK)$-linearly closed clonoid is generated by its subset of unary functions.
	
	\begin{Thm}
		\label{Thm14-2}
		Let $\KK = \prod_{i=1}^m\mathbb{F}_{q_i}$, $\FF= \prod_{i=1}^s\mathbb{F}_{p_i}$ be products of fields such that $|\KK|$ and $|\FF|$ are coprime. Then every $(\mathbb{F},\mathbb{K})$-linearly closed clonoid is generated by a set of unary functions and thus there are finitely many distinct $(\FF,\KK)$-linearly closed clonoids.
	\end{Thm}
	
	The proof of this result is given in Section~\ref{AllGen2}. From this follows that under the assumptions of  Theorem \ref{Thm14-2} we can bound the cardinality of the lattice of all $(\FF,\KK)$-linearly closed clonoids.
	
	Furthermore, in Section~\ref{TheLattice-2} we find a description of the lattice of all $(\FF,\KK)$-linearly closed clonoids as the direct product of the lattices of all $\FF_{p_i}[\KK^{\times}]$-submodules of $\mathbb{F}_{p_i}^{\KK}$, where $\KK^{\times}$ is the multiplicative monoid of $\KK = \prod_{i=1}^m \FF_{q_i}$. Moreover, we provide a concrete bound for the cardinality of the lattice of all $(\FF,\KK)$-linearly closed clonoids.
	
	\begin{Thm}
		
		\label{Corfinale-3}
		Let $\FF = \prod_{i=1}^s\FF_{p_i}$ and $\KK = \prod_{j=1}^m\FF_{q_j}$ be products of finite fields  such that $|\KK|$ and $|\FF|$ are coprime. Then the cardinality of the lattice of all $(\mathbb{F},\mathbb{K})$-linearly closed clonoids $\mathcal{L}(\FF,\KK)$  is bounded by:
		
		\begin{equation*}
			|\mathcal{L}(\FF,\KK)| \leq \prod_{i=1}^s\sum_{1 \leq r \leq n}{{ n}\choose{r}}_{p_i},
		\end{equation*}
		where $n = \prod_{j = 1}^mq_i$ and 
		\begin{equation*}
			{{n}\choose{h}}_q = \prod_{i=1}^h \frac{q^{n-h+i}-1}{q^i-1}
		\end{equation*}
		with $q \in \NN\backslash\{1\}$.
	\end{Thm}

	\section{Preliminaries and notation}\label{Preliminaries-2}
	We use boldface letters for vectors, e. g., $\mathbfsl{u} = (u_1,\dots,u_n)$ for some $n \in \NN$. Moreover, we will use $\langle\mathbfsl{v}, \mathbfsl{u}\rangle$ for the scalar product of the vectors $\mathbfsl{v}$ and $\mathbfsl{u}$.
	Let $f$ be an $n$-ary function from an additive group $\mathbf{G}_1$ to a group $\mathbf{G}_2$. We say that $f$ is \emph{$0$-preserving} if $f(0_{\mathbf{G}_1},\dots,0_{\mathbf{G}_1})	= 0_{\mathbf{G}_2}$. 
	A non-trivial $(\FF,\KK)$-linearly closed clonoids is the set of all $0$-preserving finitary functions from $\KK$ to $\FF$. The $(\FF,\KK)$-linearly closed clonoids form a lattice with the intersection as meet and the $(\FF,\KK)$-linearly closed clonoid generated by the union as join. The top element of the lattice is the $(\FF,\KK)$-linearly closed clonoid of all functions and the bottom element consists of only the constant zero functions. We write \index{$\Cid(S)$}$\Cid(S)$ for the $(\FF,\KK)$-linearly closed clonoid generated by a set of functions $S$. 
	
	In order to prove Theorem \ref{Thm14-2} we introduce the definition of $0$-absorbing function. This concept is slightly different from the one in \cite{Aic.SSOE} since we consider the source set to be split into a product of sets. Nevertheless, some of the techniques in \cite{Aic.SSOE} can be used also with our definition of $0$-absorbing function.
	
	Let $A_1, \dots, A_m$ be sets, let $0_{A_i} \in A_i$, and let $J \subseteq [m]$. For all $\mathbfsl{a} = (a_1,\dots,a_m) \in \prod_{i=1}^mA_i$ we define \index{$\mathbfsl{a}^{(J)}$}$\mathbfsl{a}^{(J)}\in \prod_{i=1}^mA_i$ by $\mathbfsl{a}_i^{(J)}  =a_i$ for $i \in J$ and $(\mathbfsl{a}^{(J)})_i  = 0_{A_i}$ for $i \in [m] \backslash J$.
	
	Let $A_1, \dots, A_m$ be sets, let $0_{A_i} \in A_i$, let $\mathbf{G} = \langle G, +, -, 0_G \rangle$ be an abelian group, let $f \colon \prod_{i=1}^mA_i \rightarrow G$, and let $I \subseteq [m]$. By \index{$\Dep(f)$}$\Dep(f)$ we denote $\{i \in [m] \mid f \text{ depends on its $i$th set argument}\}$. We say that $f$ is $0_{A_j}$-\emph{absorbing} in its $j$th argument if for all $\mathbfsl{a} = (a_1, \dots , a_m) \in \prod_{i=1}A_i$ with $a_j = 0_{A_j}$ we have $f(\mathbfsl{a}) = 0_G$. We say that $f$ is $0$-\emph{absorbing} in $I$ if $\Dep(f) \subseteq I$ and for every $i \in I$ $f$ is $0_{A_i}$-absorbing in its $i$th argument.
	
	Using the same proof of  \cite[Lemma $3$]{Aic.SSOE} we can find an interesting property of $0$-absorbing functions.
	
	\begin{Lem}
		\label{Lem0absor}
		Let $A_1, \dots, A_m$ be sets, let $0_{A_i}$ be an element of $A_i$ for all $i \in [m]$. Let $B = \langle B, +, -, 0_G\rangle$ be an abelian group, and let $f \colon \prod_{i=1}^mA_{i} \rightarrow B$. Then there is exactly one sequence $\{f_I\}_{I \subseteq [m]}$ of functions from $\prod_{i=1}^mA_{i} $ to $B$ such that for each $I \subseteq [m]$, $f_I$ is $0$-absorbing in $I$ and $f = \sum_{I \subseteq [m]} f_I$ . Furthermore, each function $f_I$ lies in the subgroup $\mathbf{F}$ of $\mathbf{B}^{\prod_{i=1}^mA_{i}}$ that is generated by the functions $\mathbfsl{x} \rightarrow f(\mathbfsl{x}^{(J)})$, where $J\subseteq [m]$.
	\end{Lem}
	
	\begin{proof}
		The proof is essentially the same of \cite[Lemma $3$]{Aic.SSOE} substituting $A^m$ with $\prod_{i =1}^mA_i$.  We define $f_I$ by recursion on $|I|$. We define $f_{\emptyset} (\mathbfsl{a}) := f(0_{A_1}, \dots , 0_{A_m})$ and for all $I \not=\emptyset$ and $\mathbfsl{a} \in \prod_{i=1}^mA_i$ and $f_{I}$ by:
		
		\begin{equation}
			f_I(\mathbfsl{a}) := f(\mathbfsl{a}^{(I)}) - \sum_{J \subset I} f_J (\mathbfsl{a}),
		\end{equation}	
		for all $\mathbfsl{a} \in \prod_{i=1}^mA_i$. 
	\end{proof}
	
	Furthermore, as in \cite{AM.CWTR}, we can see that the component $f_I$ satisfies $f_I(\mathbfsl{a})$ $ = \sum_{J\subseteq I} (-1)^{|I|+|J|}$ $f(\mathbfsl{a}^{(J)} )$. From now on we will not specify the element that the functions absorb since it will always be the $0$ of a finite field.
	
	\section{Unary generators of $(\mathbb{F},\mathbb{K})$-linearly closed clonoid}\label{AllGen2}
	
	In this section our aim is to find an analogon of  \cite[Theorem $4.2$]{Fio.CSOF} for a generic $(\mathbb{F},\mathbb{K})$-linearly closed clonoid $C$, which will allow us to generate $C$ with a set of unary functions. In general we will see that it is the unary part of an $(\mathbb{F},\mathbb{K})$-linearly closed clonoid that determines the clonoid. To this end we shall show the following lemmata. We denote by \index{$\mathbfsl{e}_1^{\FF^n_{q_i}}$}$\mathbfsl{e}_1^{\FF^n_{q_i}} = (1,0,\dots,0)$ the first member of the canonical basis of $\FF^n_{q_i}$ as a vector-space over $\FF_{q_i}$. Let $f: \prod_{i=1}^m \FF^k_{q_i} \rightarrow  \FF_{p}$. Let $s \leq m$ and let $\KK = \prod_{i=1}^s \FF_{q_i}$. Then we denote by $f\mid_{\KK}: \prod_{i=1}^s \FF^k_{q_i} \rightarrow \FF_{p}$ the function such that $f\mid_{\KK}(\mathbfsl{x}_1,\dots,\mathbfsl{x}_s) = f(\mathbfsl{x}_1,\dots,\mathbfsl{x}_s,0,\dots,0)$. 
	
	\begin{Lem}
		\label{Lem1-2}
		Let $f,g\colon \prod_{i=1}^m\mathbb{F}_{q_i}^n  \to \mathbb{F}_p$ be functions, and let $\mathbfsl{b}_1, \dots, \mathbfsl{b}_m$ be such that $\mathbfsl{b}_i \in \mathbb{F}_{q_i}^n \backslash$ $\{(0,\dots,0)\}$  for all $i \in [m]$. Assume that $f(\lambda_1\mathbfsl{b}_1,\dots,\lambda_m\mathbfsl{b}_m) =  g(\lambda_1\mathbfsl{e}_1^{\FF^n_{q_1}},\dots,\lambda_m\mathbfsl{e}_1^{\FF^n_{q_m}})$, for all $\lambda_1 \in \mathbb{F}_{q_1},\dots,\lambda_m \in \mathbb{F}_{q_m}$, and $f(\mathbfsl{x}) = g(\mathbfsl{y}) = 0$ for all $\mathbfsl{x} \in \prod_{i=1}^m\mathbb{F}_{q_i}^n \backslash \{(\lambda_1\mathbfsl{b}_1,\dots$ $,\lambda_m\mathbfsl{b}_m) \mid (\lambda_1,\dots,\lambda_m) \in \prod_{i=1}^m\mathbb{F}_{q_i}\}$ and $\mathbfsl{y} \in \prod_{i=1}^m\mathbb{F}_{q_i}^n \backslash \{(\lambda_1\mathbfsl{e}_1^{\FF^n_{q_1}},\dots,\lambda_m\mathbfsl{e}_1^{\FF^n_{q_m}})\mid  (\lambda_1,\dots,$ $\lambda_m) \in \prod_{i=1}^m\mathbb{F}_{q_i}\}$. Then $f \in \Cid(\{g\})$.
	\end{Lem}
	
	\begin{proof}
		For $j \leq m$ let $A_j$ be any invertible $n \times n$-matrix  over $\KK_j$  such that  $A_j\mathbfsl{b}_j = \mathbfsl{e}^{\FF^n_{q_j}}_1$. Then is straighforward to check that $f(\mathbfsl{x}_1,\dots,\mathbfsl{x}_m) =
		g(A\mathbfsl{x}_1,$ $\dots,A\mathbfsl{x}_m)$.
	\end{proof}

	\begin{Lem}
		\label{Rem-sum-0-pres2}
		Let $q_1,\dots,q_m$ and $p$ be powers of  primes and let $\KK = \prod_{i=1}^m\FF_{q_i}$. Let $h \leq m$ and let $\KK_1 = \prod_{i=1}^h\FF_{q_i}$. Let $C$ be an $(\FF_p,\KK)$-linearly closed clonoid and let 
		\begin{equation*}
			C\mid_{\KK_1} := \{g \mid \exists g' \in C \colon g'\mid_{\KK_1}  = g\}.
		\end{equation*}
		Let $\Dep(f) = [h]$ and let $f: \prod_{i=1}^m\FF^s_{q_i} \rightarrow \FF_p$. Then $f \in \Cid(C^{[1]})^{[s]}$ if and only if $f \mid_{\KK_1} \in \Cid(C\mid_{\KK_1}^{[1]})^{[s]}$.
	\end{Lem}
	
	\begin{proof}
		It is clear that if $f \in \Cid(C^{[1]})$ then $f \mid_{\KK_1} \in \Cid(C\mid_{\KK_1}^{[1]})$, simply restricting to $\KK_1$ all the unary generators of $f$. Conversely, let $S'$ be a set of unary generators of  $f\mid_{\KK_1}$. Let $S \subseteq C^{[1]}$ be defined by
		\begin{align*}
			S := \{&g \mid \exists g' \in S' : g(x_1,\dots,x_h,0,\dots,0) = g'(x_1,\dots,x_h), \\&\text{ for all } (x_1,\dots,x_h) \in\prod_{i=1}^h\FF_{q_i} \}.
		\end{align*}
		From $\Dep(f) = [h]$ follows that $S$ is a set of unary generators of $f$.
	\end{proof}
	
	\begin{Lem}
		\label{Lemt_k}
		Let $q_1,\dots,q_m$ and $p$ be powers of  primes with $\prod_{i=1}^mq_i$ and $p$ coprime. Let $\KK = \prod_{i=1}^m\FF_{q_i}$. Let $C$ be an $(\FF_p,\KK)$-linearly closed clonoid, let $g \in C^{[1]}$ be $0$-absorbing in $[m]$, and let $t_k\colon \prod_{i=1}^m\mathbb{F}_{q_i}^k \to \mathbb{F}_p$ be defined by:
		
		\begin{align*}
			&t_k(\lambda_1\mathbfsl{e}_1^{\FF^k_{q_1}},\dots,\lambda_m\mathbfsl{e}_1^{\FF^k_{q_m}}) = g(\lambda_1,\dots,\lambda_m) \text{ for all } (\lambda_1,\dots,\lambda_m) \in \prod_{i = 1}^m\FF_{q_i} 
			\\&t_k(\mathbfsl{x}) = 0 
			\\&\text{ for all } \mathbfsl{x} \in \prod_{i=1}^m\mathbb{F}^k_{q_i}\backslash \{( \lambda_1\mathbfsl{e}_1^{\FF^k_{q_1}}, \dots,\lambda_m\mathbfsl{e}_1^{\FF^k_{q_m}}) \mid  (\lambda_1,\dots,\lambda_m) \in \prod_{i = 1}^m\FF_{q_i}\}.
		\end{align*}
		Then $t_k$ is $0$-absorbing in $[m]$, with $A_i = \FF_{q_i}^k$ and $0_{A_i} = (0_{\FF_{q_i}},\dots,0_{\FF_{q_i}})$. Furthermore,  $t_k \in \Cid(C^{[1]})$ for all $k \in \NN$.
	\end{Lem}
	
	\begin{proof}
		Since $g$ is $0$-absorbing in $[m]$  then also $t_k$ is $0$-absorbing in $[ m]$.
		
		Moreover ,we prove that $t_k \in \Cid(C^{[1]})$ by induction on $k$.
		
		Case $k =1$: if $k = 1$, then $t_1 = g$ is a unary function of $C^{[1]}$. 
		
		Case $k>1$: we assume that $t_{k-1} \in \Cid(C^{[1]})$. 
		
		For all $1 \leq i \leq m$ we define the two sets of mappings $T_i^{[k]} $ and $R_i^{[k]}$ from $\FF_{q_i}^k$ to $\FF_{q_i}^{k-1}$ by:
		\begin{align*}
			T_i^{[k]} &:=\{u_{a}\colon (x_1,\dots,x_k) \mapsto (x_1-ax_2,x_3\dots,x_k)\mid a \in \mathbb{F}_{q_i}\}
			\\R_i^{[k]} &:=\{w_{a}\colon (x_1,\dots,x_k) \mapsto (ax_2,x_3\dots,x_k)\mid a \in \mathbb{F}_{q_i}\backslash\{0\}\}.
		\end{align*}
		Let $P_i^{[k]} := T_i^{[k]} \cup R_i^{[k]}$. Furthermore, we define the function $c ^{[k]}\colon \bigcup_{i=1}^m P_i^{[k]} \rightarrow \NN$ by:
		\begin{align*}
			c^{[k]}(h)= \begin{cases} 0 & \text{if } h \in \bigcup_{i=1}^mT_i^{[k]} \\
				1 & \text{if } h \in \bigcup_{i=1}^mR_i^{[k]}. \end{cases}
		\end{align*}
		Let us define the function $r_k\colon \prod_{i=1}^m\mathbb{F}_{q_i}^k \rightarrow \mathbb{F}_p$ by:
		
		\begin{align}\label{eq4-2}
			&r_k(\mathbfsl{x}_1,\dots, \mathbfsl{x}_m) = \nonumber \\& = \sum_{h_1 \in P_1^{[k]},\dots,h_m \in P_m^{[k]}} (-1)^{\sum_{i=1}^m \text{c}^{[k]}(h_i)} t_{k-1}(h_1(\mathbfsl{x}_1), ,\dots,h_m(\mathbfsl{x}_m)),
		\end{align}
		for all $\mathbfsl{x}_i \in \mathbb{F}_{q_i}^k$. 
		
		\textbf{Claim:} $r_k(\mathbfsl{x}_1,\dots, \mathbfsl{x}_m) = \prod_{i =1}^mq_i \cdot t_{k}(\mathbfsl{x}_1,\dots, \mathbfsl{x}_m)$ for all $(\mathbfsl{x}_1,\dots, \mathbfsl{x}_m) \in \prod_{i=1}^m \FF_{q_i}^k$
		
		Subcase $\exists i \in [m], 3 \leq j \leq k$ with $(\mathbfsl{x}_i)_j \not=0$:
		
		By definition of $t_{k-1}$, we can see that in \eqref{eq4-2} every summand vanishes if there exist $i \in [m]$ and $3 \leq j \leq k$ with $(\mathbfsl{x}_i)_j \not=0$. Thus $r_k(\mathbfsl{x}_1,\dots, \mathbfsl{x}_m) = \prod_{i =1}^mq_i \cdot t_{k}(\mathbfsl{x}_1,\dots, \mathbfsl{x}_m) = 0$ in this case.
		
		Subcase $\exists l \in [m] $ with $(\mathbfsl{x}_l)_2 \not=0$ and $(\mathbfsl{x}_i)_j =0$ for all $ i \in [m], 3 \leq j \leq k$:
		
		We prove that $r_k(\mathbfsl{x}_1,\dots, \mathbfsl{x}_m) = 0$. We can see that for all $(x_1,x_2) \in \FF_{q_l} \times \FF_{q_l}\backslash\{0\}$ and for all $b \in \FF_{q_l}\backslash\{0\}$, there exists $a \in \FF_{q_l}$ such that $bx_2 = x_1-ax_2$, and clearly $a = x_1x_2^{-1}-b$. Conversely, for all $(x_1,x_2) \in \FF_{q_l} \times \FF_{q_l}\backslash\{0\}$ and for all $a \in \FF_{q_l}\backslash\{x_1x_2^{-1}\}$ there exists $b \in \FF_{q_l}\backslash\{0\}$ such $bx_2 = x_1-ax_2$, and clearly $b = x_1x_2^{-1}-a$.
		
		With this observation we can see that for all $h_i \in P_i^{[k]}$ with $i \in [m]\backslash\{l\}$ and for all $(\mathbfsl{x}_1,\dots,\mathbfsl{x}_m) \in \prod_{i=1}^m\FF^k_{q_i}$ with $(\mathbfsl{x}_l)_1 = x_1$ and $(\mathbfsl{x}_l)_2 = x_2$ we have that if $a \not= x_1x_2^{-1}$ then:
		
		\begin{align*}
			\label{eqrl1}
			&t_{k-1} (h_1(\mathbfsl{x}_1),\dots,h_{l-1}(\mathbfsl{x}_{l-1}),u_{a}(\mathbfsl{x}_l),h_{l+1}(\mathbfsl{x}_{l+1}),\dots,h_m(\mathbfsl{x}_{m}))= 
			\\&=t_{k-1} (h_1(\mathbfsl{x}_1),\dots,h_{l-1}(\mathbfsl{x}_{l-1}),w_{x_1x_2^{-1}-a}(\mathbfsl{x}_l),h_{l+1}(\mathbfsl{x}_{l+1}),\dots,h_m(\mathbfsl{x}_{m}))
		\end{align*}	
		where $u_{a} \in T_l^{[k]}$ and $w_{x_1x_2^{-1}-a}\in R_l^{[k]}$. Thus they produce summands with different signs in \eqref{eq4-2}. Moreover, if $a = x_1x_2^{-1}$, then
		
		\begin{align*}
			&t_{k-1} (h_1(\mathbfsl{x}_1),\dots,h_{l-1}(\mathbfsl{x}_{l-1}),u_{a}(\mathbfsl{x}_l),h_{l+1}(\mathbfsl{x}_{l+1}),\dots,h_m(\mathbfsl{x}_{m})) = 
			\\&= t_{k-1} (h_1(\mathbfsl{x}_1),\dots,h_{l-1}(\mathbfsl{x}_{l-1}),\mathbfsl{0}_{\FF^{k-1}_{q_l}},h_{l+1}(\mathbfsl{x}_{l+1}),\dots,h_m(\mathbfsl{x}_{m})) = 0, 
		\end{align*}
		since $t_{k-1}$ is $0$-absorbing in $[m]$. This implies that all the summands of $r_k$ are cancelling if $(\mathbfsl{x}_l)_2 \not= 0$. Thus $r_k(\mathbfsl{x}_1,\dots, \mathbfsl{x}_m) = \prod_{i =1}^mq_i \cdot t_{k}(\mathbfsl{x}_1,\dots, \mathbfsl{x}_m) = 0$ in this case.
		
		Subcase $(\mathbfsl{x}_1,\dots,\mathbfsl{x}_m) = (\lambda_1\mathbfsl{e}_1^{\FF^{k}_{q_1}},\dots, \lambda_m\mathbfsl{e}_1^{\FF^{k}_{q_m}})$ for some $(\lambda_1,\dots,\lambda_m) \in \prod_{i=1}^m \FF_{q_i}$:
		
		We can observe that:
		\begin{align*}
			&t_{k-1}  (h_1(\mathbfsl{x}_1),\dots,h_{l-1}(\mathbfsl{x}_{l-1}),h_l(\lambda_l\mathbfsl{e}_1^{\FF_{q_l}}),h_{l+1}(\mathbfsl{x}_{l+1}),\dots,h_m(\mathbfsl{x}_{m})) = 0
			\\&=t_{k-1}  (h_1(\mathbfsl{x}_1),\dots,h_{l-1}(\mathbfsl{x}_{l-1}),\mathbfsl{0}_{\FF^{k-1}_{q_l}},h_{l+1}(\mathbfsl{x}_{l+1}),\dots,h_m(\mathbfsl{x}_{m})) = 0, 
		\end{align*}
		for all $h_i \in P_i^{[k]}$ with $i \in [m]\backslash\{l\}$, for all $l \leq m$, $\lambda_l \in \FF_{q_l}$, $\mathbfsl{x}_i \in \FF^k_{q_i}$, and $h_{l} \in R^{[k]}_l$, since $t_{k-1}$ is $0$-absorbing in $[n]$. Thus we can observe that:
		\begin{equation*}\label{eq6}
			\begin{split}
				&r_k(\lambda_1\mathbfsl{e}_1^{\FF^k_{q_1}},\dots, \lambda_m\mathbfsl{e}_1^{\FF^k_{q_m}}) = 
				\\ =&\sum_{h_i \in P_i^{[k]}}  (-1)^{\sum_{i=1}^m \text{c}^{[k]}(h_i)}t_{k-1} (h_1(\lambda_1\mathbfsl{e}_1^{\FF^k_{q_1}}),\dots, h_m(\lambda_m\mathbfsl{e}_1^{\FF^k_{q_m}}))
				\\=& \sum_{h_i \in T_i^{[k]}}  (-1)^{\sum_{i=1}^m \text{c}^{[k]}(h_i)}t_{k-1} (h_1(\lambda_1\mathbfsl{e}_1^{\FF^k_{q_1}}),\dots, h_m(\lambda_m\mathbfsl{e}_1^{\FF^k_{q_m}}))
				\\=& \sum_{h_i \in T_i^{[k]}} t_{k-1} (h_1(\lambda_1\mathbfsl{e}_1^{\FF^k_{q_1}}),\dots, h_m(\lambda_m\mathbfsl{e}_1^{\FF^k_{q_m}}))
				\\=& \sum_{h_i \in T_i^{[k]}} t_{k-1} (\lambda_1\mathbfsl{e}_1^{\FF_{q_1}^{k-1}},\dots, \lambda_m\mathbfsl{e}_1^{\FF_{q_m}^{k-1}})
				\\=&\prod_{i=1}^mq_i\cdot t_{k-1} (\lambda_1\mathbfsl{e}_1^{\FF_{q_1}^{k-1}},\dots, \lambda_m\mathbfsl{e}_1^{\FF_{q_m}^{k-1}}) 
				\\ = & \prod_{i=1}^mq_i\cdot t_k(\lambda_1\mathbfsl{e}_1^{\FF^{k}_{q_1}},\dots, \lambda_m\mathbfsl{e}_1^{\FF^{k}_{q_m}}).
			\end{split}
		\end{equation*}
		Thus $r_k = \prod_{i =1}^mq_i \cdot t_{k}$.
		
		Because of \eqref{eq4-2} and the inductive hypothesis, we have $r_k \in \Cid(\{t_{k-1}\})$ $ \subseteq \Cid(C^{[1]})$. Thus $\prod_{i = 1}^mq_it_{k} \in \Cid(C^{[1]})$. Since $\prod_{i = 1}^mq_i \not= 0$ modulo $p$ we have that $t_{k} \in \Cid(C^{[1]})$ and this concludes the induction proof. 
	\end{proof}
	
	\begin{Lem}
		\label{Lem-sum-0-pres}
		Let $q_1,\dots,q_m$ and $p$ be powers of  primes with $\prod_{i=1}^mq_i$ and $p$ coprime and let $\KK = \prod_{i=1}^m\FF_{q_i}$. Let $C$ be an $(\FF_p,\KK)$-linearly closed clonoid, let $I \subseteq [m]$ and let $f \in C$ be $0$-absorbing in $I$. Then $f \in \Cid(C^{[1]})$.
	\end{Lem}
	
	\begin{proof}
		
		Let $\KK_1 = \prod_{i \in I}\FF_{q_i}$ and let $C_1 := \{g \mid \exists g' \in C \colon g'\mid_{\KK_1}  = g\}$. By Lemma \ref{Rem-sum-0-pres2} $f \in \Cid(C^{[1]})$ if and only if $f \mid_{\KK_1} \in \Cid(C\mid_{\KK_1}^{[1]})$ and we observe that $f\mid_{\KK_1}$ is $0$-absorbing in $I$ . Thus without loss of generality we fix $I = [m]$. The strategy is to interpolate $f$ in all the distinct products of lines of the form $\{(\lambda_{1}\mathbfsl{b}_{1},\dots,\lambda_m\mathbfsl{b}_{m}) \mid  (\lambda_1,\dots,\lambda_m) \in \prod_{i = 1}^m\FF_{q_i}, \mathbfsl{b}_i \in \FF_{q_i}^n\backslash\{(0,\dots,0)\}$. To this end let  $R = \{L_j \mid 1 \leq j \leq \prod_{i = 1}^m(q_i^{n} - 1)/(q_i-1) = s\}$ be the set of all $s$ distinct products of lines of $\prod_{i = 1}^m\FF_{q_i}$ and let $\mathbfsl{l}_{(i,j)} \in \mathbb{F}_{q_i}^n$ be such that $(\mathbfsl{l}_{(1,j)},\dots,\mathbfsl{l}_{(m,j)})$ generates the products of $m$ lines $L_j$, for $1 \leq j \leq s$, $1 \leq i \leq m$. For all $1 \leq j \leq s$, let $f_{L_j}\colon \prod_{i=1}^m\mathbb{F}_{q_i}^n \to \mathbb{F}_p$ be defined by:
		
		\begin{equation*} f_{L_j}(\lambda_1\mathbfsl{l}_{(1,j)},\dots,\lambda_m\mathbfsl{l}_{(m,j)}) = f(\lambda_1\mathbfsl{l}_{(1,j)},\dots,\lambda_m\mathbfsl{l}_{(m,j)}) 
		\end{equation*}
		for $ (\lambda_1,\dots,\lambda_m) \in \prod_{i = 1}^m\FF_{q_i}$ and $f_{L_j}(\mathbfsl{x}) = 0$ for all $\mathbfsl{x} \in \prod_{i=1}^m\mathbb{F}^n_{q_i}\backslash \{(\lambda_1\mathbfsl{l}_{(1,j)},$ $\dots,\lambda_m\mathbfsl{l}_{(m,j)}) \mid  (\lambda_1,\dots,\lambda_m) \in \prod_{i = 1}^m\FF_{q_i}\}$.
		
		\textbf{Claim 1:}
		\begin{equation*}
			f= \sum_{j = 1}^{s} f_{L_j}.
		\end{equation*}
		Since $f$ is $0$-absorbing in $[m]$ we have that: 
		
		\begin{align*}
			\sum_{j = 1}^{s} f_{L_j}(\lambda_1\mathbfsl{l}_{(1,z)},\dots,\lambda_m\mathbfsl{l}_{(m,z)}) &= f_{L_z}(\lambda_1\mathbfsl{l}_{(1,z)},\dots,\lambda_m\mathbfsl{l}_{(m,z)})=
			\\&=f(\lambda_1\mathbfsl{l}_{(1,z)},\dots,\lambda_m\mathbfsl{l}_{(m,z)}) 
		\end{align*}
		for all $ (\lambda_1,\dots,\lambda_m) \in \prod_{i = 1}^m\FF_{q_i}$ and $z \in [s]$, since for all $j_1,j_2 \in [s]$, $L_{j_1}$ and $L_{j_2}$ intersect only in points of the form $(\mathbfsl{x}_1,\dots,\mathbfsl{x}_m)$ $ \in \prod_{i = 1}^m\FF_{q_i}^n$ with $\mathbfsl{x}_i = (0,\dots,0)$ for some $i \in [m]$. 
		
		Let $1 \leq j \leq s$ and let $g\colon \prod_{i=1}^m\mathbb{F}_{q_i} \to \mathbb{F}_p$ be a function such that:
		\begin{equation*}
			f_{L_j}(\lambda_1\mathbfsl{l}_{(1,j)},\dots,\lambda_m\mathbfsl{l}_{(m,j)}) = g(\lambda_1,\dots,\lambda_m) = f(\lambda_1\mathbfsl{l}_{(1,j)},\dots,\lambda_m\mathbfsl{l}_{(m,j)})
		\end{equation*}
		for all $ (\lambda_1,\dots,\lambda_m) \in \prod_{i = 1}^m\FF_{q_i}$.  Then $g \in C^{[1]}$.
		
		\textbf{Claim 2:} $f_{L_j} \in \Cid(C^{[1]})$ for all $L_j \in R$. 
		
		We can observe that $f_{L_j}(\lambda_1\mathbfsl{l}_{(1,j)},\dots, \lambda_m\mathbfsl{l}_{(m,j)}) = g(\lambda_1,\dots,\lambda_m)$ for all $(\lambda_1,\dots,$ $\lambda_m) \in \prod_{i=1}^m\mathbb{F}_{q_i}$, and $f_{L_j}(\mathbfsl{x}_1,\dots,\mathbfsl{x}_m)  = 0$ for all $(\mathbfsl{x}_1,\dots,\mathbfsl{x}_m) \in \prod_{i=1}^m\mathbb{F}_{q_i}^n \backslash \{(\lambda_1\mathbfsl{l}_{(1,j)},$ $\dots, \lambda_m\mathbfsl{l}_{(m,j)})\mid (\lambda_1,\dots,\lambda_m) \in \prod_{i=1}^m\mathbb{F}_{q_i} \}$. Furthermore, $g$ is $0$-absorbing in $[m]$. By Lemmata \ref{Lem1-2} and \ref{Lemt_k}, $f_{L_j} \in \Cid(C^{[1]})$, which concludes the proof of $f \in \Cid(C^{[1]})$. 
	\end{proof}
	
	We are now ready to prove that an $(\mathbb{F},\mathbb{K})$-linearly closed clonoid $C$ is generated by its unary part.
	
	\begin{Thm}
		\label{Th1-2}	
		Let $q_1,\dots,q_m$ and $p$ be powers of  primes with $\prod_{i=1}^mq_i$ and $p$ coprime and let $\KK = \prod_{i=1}^m\FF_{q_i}$. Then every $(\FF_p,\KK)$-linearly closed clonoid $C$ is generated by its unary functions. Thus $C = \Cid(C^{[1]})$.
	\end{Thm}
	
	\begin{proof}
		The inclusion $\supseteq$ is obvious. For the other inclusion let $C$ be an $(\FF_p,\KK)$-linearly closed clonoid and let $f$ be an $n$-ary function in $C$. By Lemma \ref{Lem0absor} with $A_i = \FF_{q_i}^n$ and $0_{A_i} = (0_{\FF_{q_i}},\dots,0_{\FF_{q_i}})$, $f$ can be split in the sum of $n$-ary functions $\sum_{I \subseteq [m]}f_I$ such that for each $I \subseteq [m]$, $f_I$ is $0$-absorbing in $I$. Furthermore, each function $f_I$ lies in the subgroup $\mathbf{F}$ of $\mathbb{F}_p^{\KK^n}$ that is generated by the functions $\mathbfsl{x} \rightarrow f(\mathbfsl{x}^{(I)})$, where $I\subseteq [m]$ and thus each summand $f_I$ is in $C$. By Lemma \ref{Lem-sum-0-pres} each of these summands is in $\Cid(C^{[1]})$. and thus $f \in  \Cid(C^{[1]})$.
	\end{proof}
	
	The next corollary of Theorem \ref{Th1-2} and the following theorem tell us that there are only finitely many distinct $(\mathbb{F},\mathbb{K})$-linearly closed clonoids.
	
	\begin{Cor}
		\label{Cor2-2}
		Let $q_1,\dots,q_m$ and $p$ be powers of  primes with $\prod_{i=1}^mq_i$ and $p$ coprime and let $\KK=\prod_{i=1}^m\FF_{q_i}$. Let $C$ and $D$ be two $(\FF_p,\KK)$-linearly closed clonoids. Then $C = D$ if and only if $C^{[1]} = D^{[1]}$.
	\end{Cor}
	
	Let us denote by $\mathcal{L}(\FF,\KK)$ the lattice of all $(\FF,\KK)$-linearly closed clonoids. We define the functions $\rho_i\colon \mathcal{L}(\FF,\KK) \rightarrow \mathcal{L}(\FF_{p_i},\KK)$ such that for all $1 \leq i \leq s$ and for all $C \in \mathcal{L}(\FF,\KK)$:
	
	\begin{equation}
		\label{defiso}
		\rho_i(C) := \{f \mid \text{ there exists } g \in C \colon f = \pi_i^{\FF} \circ g\},
	\end{equation}
	where with $\pi_i^{\FF}$ we denote the projection over the $i$-th component of the product of fields $\FF$.
	
	\begin{Thm}
		\label{ThmDirProd}
		Let $\mathbb{F} = \prod_{i =1}^s\FF_{p_i}$ and $\mathbb{K} = \prod_{i =1}^m\FF_{q_i}$ be products of finite fields. Then the lattice of all $(\FF,\KK)$-linearly closed clonoids is isomorphic to the direct product of the lattices of all $(\FF_{p_i},\KK)$-linearly closed clonoids with $1 \leq i \leq s$.
	\end{Thm}
	
	\begin{proof}
		
		Let us define the function $\rho\colon \mathcal{L}(\FF,\KK) \rightarrow \prod_{i=1}^s\mathcal{L}(\FF_{p_i},\KK)$ such that $\rho(C) := (\rho_1(C),\dots,\rho_s(C))$. Clearly $\rho$ is well-defined. Conversely, let $\psi\colon $ $\prod_{i=1}^s\mathcal{L}$ $(\FF_{p_i},\KK) \rightarrow \mathcal{L}(\FF,\KK)$ be defined by:
		\begin{equation*}
			\psi(C_1,\dots,C_s) = \bigcup_{k \in \NN}\{f \colon\mathbfsl{x} \mapsto (f_1(\mathbfsl{x}),\dots,f_s(\mathbfsl{x}))\mid f_1\in C_1^{[k]}, \dots, f_s\in C_s^{[k]}\}.
		\end{equation*}
		From this definition it is clear that $\psi$ is well defined. Furthermore,
		\begin{equation*}
			\rho\psi(C_1,\dots,C_s) = (C_1,\dots,C_s)
		\end{equation*}
		and $C\subseteq\psi\rho(C)$ for all $(C_1,\dots,C_s) \in \prod_{i=1}^s\mathcal{L}(\FF_{p_i},\KK)$ and $C \in \mathcal{L}(\FF,\KK)$.
		
		To prove that $C \supseteq \psi\rho(C)$  let $f\in \psi \rho (C)$. Then there exists $(f_1,\dots,f_s) \in \rho (C)$ such that $f_i = \pi_i^{\FF} \circ f$ for all $i \in [s]$. By definition of $\rho$, there exist $g_1,\dots,g_s \in C$ such that $f_i = \pi_i^{\FF} \circ g_i$ for all $i \in [s]$. Let $\mathbfsl{a}_i  \in  \FF$ be such that $\mathbfsl{a}_i(j) = 0$ for $j \not=  i$ and $\mathbfsl{a}_i(i) = 1$. It is easy to check that the function $f =\sum_{i=1}^s \mathbfsl{a}_ig_i =f$ and thus $f \in C$. 
		
		Hence $\rho$ is a lattice isomorphism.
		
	\end{proof}
	
	\begin{proof}[Proof of Theorem \ref{Thm14-2}]
		Let $\FF= \prod_{i=1}^s\FF_{p_i}$ and $\KK = \prod_{i=1}^m\FF_{q_i}$ be products of finite with $|\KK|$ and $|\FF|$ coprime. Let $C \in \mathcal{L}(\FF,\KK)$. By  Theorem \ref{ThmDirProd}  $C$ is uniquely determined by its projections $C_1 = \rho_1(C),\dots, C_s = \rho_s(C)$ where $\rho_i$ is defined in 	\eqref{defiso}. By Theorem \ref{Th1-2} we have that for all $i \in  [s]$ every $(\FF_{p_i},\KK)$-linearly closed clonoid $C_i$ is uniquely determined by its unary part $C_i^{[1]}$. Thus $C$ is uniquely determined by its unary part $C^{[1]}$.
	\end{proof}
	
	\section{The lattice of all $(\mathbb{F},\mathbb{K})$-linearly closed clonoids}\label{TheLattice-2}
	
	In this section we characterize the structure of the lattice $\mathcal{L}(\FF,\KK)$ of all $(\mathbb{F},\mathbb{K})$-linearly closed clonoids through a description of their unary parts. Let $\mathbb{F} = \prod_{i=1}^s\FF_{p_i}$ and $\mathbb{K} = \prod_{j=1}^m\FF_{q_j}$ be products of finite fields such that $|\KK|$ and $|\FF|$ are coprime numbers. 
	
	We will see that $\mathcal{L}(\FF,\KK)$ is isomorphic to the product of the lattices of all $\FF_{p_i}[\KK^{\times}]$-submodules of $\mathbb{F}_{p_i}^{\KK}$, where $\KK^{\times}$ is the multiplicative monoid of $\KK = \prod_{i=1}^m \FF_{q_i}$.
	In order to characterize the lattice of all $(\FF, \KK)$-linearly closed clonoids we need the definition of  \emph{monoid ring}.
	
	\begin{Def} Let $\langle M, \cdot\rangle$ be a commutative monoid and let $\langle R, +, \odot\rangle$ be a commutative ring with identity. Let
		\begin{equation*}
			S := \{f \in R^M \mid f(a) \not= 0 \text{ for only finitely many } a \in M\}.
		\end{equation*}
	\end{Def}
	
	We define the \emph{monoid ring} of $M$ over $R$ as the ring $(S, +, \cdot)$, where $+$ is the point-wise addition of functions and the multiplication is defined as $f\cdot g : R \rightarrow M $ which maps each $m \in M$ into:
	\begin{equation*}
		\sum_{m_1,m_2 \in M,m_1m_2=m}f(m_1)g(m_2).
	\end{equation*}
	We denote by $R[M]$ the monoid ring of $M$ over $R$. Following the notation in  \cite{AM.CWTR} for all $a \in A$ we define $\tau_a$ to be the element of $R^M$ with $\tau_a(a) = 1$ and $\tau_a (M\backslash\{a\}) = \{0\}$. We observe that for all $f \in R[M]$ there is an $\mathbfsl{r} \in R^M$ such that $f = \sum_{a\in M} r_a\tau_a $ and that we can multiply such expressions with the rule $\tau_a \cdot \tau_b = \tau_{ab }$. 
	
	\begin{Def}\label{DefAct2}
		Let $M$ be a commutative monoid and let $R$ be a commutative ring. We denote by $R^M$ the $R[M]$-module with the action $\ast$ defined by:
		\begin{equation*}
			(\tau_{a} \ast  f)(x) = f(ax),
		\end{equation*}
		for all $a \in M$ and $f \in R^M$.
	\end{Def}
	
	Let $\KK^{\times}$ be the multiplicative monoid of $\KK = \prod_{i=1}^m \FF_{q_i}$. We can observe that $V$ is an $\FF_p[\KK^{\times}]$-submodule of $\mathbb{F}_p^{\KK}$ if and only if it is a subspace of $\mathbb{F}_p^{\KK}$ satisfying
	
	\begin{equation}
		\label{neweq12}
		(x_1,\dots,x_m)  \mapsto f(a_1x_1,\dots,a_mx_m) \in V,
	\end{equation}
	for all $f  \in V$ and $(a_1,\dots,a_m) \in \prod_{i=1}^m\FF_{q_i}$. Clearly the following lemma holds.
	
	\begin{Lem}
		Let $p, q_1,\dots q_m$ be powers of primes and let $\KK = \prod_{i =1}^m \FF_{q_i}$. Let  $V \subseteq \mathbb{F}_p^{\KK}$. Then $V$ is the unary part of an $(\mathbb{F}_p , \KK)$-linearly closed clonoid if and only if is an $\FF_p[\KK^{\times}]$-submodule of $\mathbb{F}_p^{\KK}$.
	\end{Lem}
	
	Together with Theorem \ref{ThmDirProd} this immediately yields the following.
	
	\begin{Cor}
		\label{CorModuleIso2}
		Let $\KK = \prod_{i =1}^m \FF_{q_i}$ and $\FF = \prod_{i =1}^s \FF_{p_i}$ be products of finite fields such that $|\KK|$ and $|\FF|$ are coprime. Then the function $\pi^{[1]}$ that sends an $(\FF , \KK)$-linearly closed clonoid to its unary part is an isomorphism between the lattice of all $(\FF, \KK)$-linearly closed clonoids and the direct product of the lattices of  all $\FF_{p_i}[\KK^{\times}]$-submodules of $\mathbb{F}_{p_i}^{\KK}$.
	\end{Cor}
	
	With the same strategy of \cite[Lemma $5.6$]{Fio.CSOF} we obtain the following Lemma.
	
	\begin{Lem}
		Let $\KK = \prod_{i =1}^m \FF_{q_i}$ and $\FF = \prod_{i =1}^s \FF_{p_i}$ be products of finite fields such that $|\KK|$ and $|\FF|$ are coprime. Then every $(\FF , \KK)$-linearly closed clonoid is finitely related.
	\end{Lem}
	
	The next step is to characterize the lattice of all $\FF_p[\KK^{\times}]$-submodules of $\mathbb{F}_p^{\KK}$. To this end we observe that $V$ is an $\FF_p[\KK^{\times}]$-submodule of $\mathbb{F}_p^{\KK}$ if and only if is a subspace of $\mathbb{F}_p^{\KK}$ satisfying \eqref{neweq12}.
	
	We can provide a bound for the lattice of all $(\FF,\KK)$-linearly closed clonoids given by the number of subspaces of  $\FF_{p_i}^{\KK}$.
	
	\begin{Rem}
		It is a well-known fact in linear algebra that the number of $k$-dimensional subspaces of an $n$-dimensional vector space $V$ over a finite field $\FF_q$ is the Gaussian binomial coefficient:
		
		\begin{equation}
			{{n}\choose{k}}_q = \prod_{i=1}^k \frac{q^{n-k+i}-1}{q^i-1}.
		\end{equation}
		
	\end{Rem}
	
	From this remark we directly obtain the bound of  Theorem \ref{Corfinale-3}. In order to determine the exact cardinality of the lattice of all $(\FF,\KK)$-linearly closed clonoids we have to deal with the problem to find the  $\FF_p[\KK^{\times}]$-submodules of $\mathbb{F}_p^{\KK}$. We will not study this problem here because we think that this is an interesting problem that deserves an own research.

	\section*{Acknowledgements}
	
	The author thanks Erhard Aichinger, who inspired this paper, Erkko Lehtonen who reviewed my Ph. D. thesis, and Sebastian Kreinecker for many hours of fruitful discussions. The author thanks the referees for their useful suggestions.

\end{document}